\documentclass[ a4, 14pt]{amsart}
\usepackage{amssymb}
\usepackage{amstext}
\usepackage{amsmath}
\usepackage{amscd}
\usepackage{latexsym}
\usepackage{amsfonts}
\usepackage{color}
\usepackage{enumerate}
\usepackage{textcomp}
\usepackage[all]{xy}
\usepackage{diagrams}
\usepackage{graphicx}
\usepackage{appendix}
\usepackage[]{algorithm2e}
\usepackage{xypic}
\usepackage{url}
\usepackage[bookmarksnumbered]{hyperref}
\hypersetup{colorlinks=true,backref=true,citecolor=blue}

\usepackage[alphabetic,initials]{amsrefs}

\DeclareFontFamily{OT1}{rsfs}{}
\DeclareFontShape{OT1}{rsfs}{n}{it}{<-> rsfs10}{}
\DeclareMathAlphabet{\mathscr}{OT1}{rsfs}{n}{it}

\theoremstyle{plain}
\newtheorem{thm}{Theorem}[section]
\newtheorem*{thm*}{Theorem}
\newtheorem*{cor*}{Corollary}
\newtheorem*{defn*}{Definition}

\newtheorem{prop}[thm]{Proposition}
\newtheorem{{expFact}}[thm]{Experimental Fact}

\newtheorem{cor}[thm]{Corollary}

\newtheorem*{claim*}{Claim}

\theoremstyle{definition}

\newtheorem{ex}[thm]{Example}
\newtheorem{rem}[thm]{{\it Remark}}

\theoremstyle{remark}

\numberwithin{equation}{thm}

\newcommand{\PP}{{\Bbb P}}
\newcommand{\GG}{{\Bbb G}}
\newcommand{\CC}{{\Bbb C}}
\newcommand{\QQ}{{\Bbb Q}}
\newcommand{\ZZ}{{\Bbb Z}}
\newcommand{\FF}{{\Bbb F}}

\newcommand{\sO}{{\mathcal O}}
\newcommand{\sI}{{\mathcal I}}
\newcommand{\sF}{{\mathcal F}}
\newcommand{\sE}{{\mathcal E}}

\newcommand{\sS}{{\mathcal S}}

\newcommand{\sEC}{{\mathcal {EC}}}

\newcommand{\sR}{{\mathcal R}}
\newcommand{\sC}{{\mathcal C}}

\newcommand{\tensor}{{\otimes}}

\newcommand{\calE}{\mathcal{E}}

\newcommand{\fkM}{\mathfrak{M}}

\newcommand{\fkP}{\mathfrak{P}}

\DeclareMathOperator{\rank}{{rank}}

\DeclareMathOperator{\Pic}{{Pic}}
\DeclareMathOperator{\coker}{{coker}}
\DeclareMathOperator{\Ext}{{Ext}}
\DeclareMathOperator{\PGL}{{PGL}}
\DeclareMathOperator{\Sec}{{Sec}}
\DeclareMathOperator{\Spec}{{Spec}}
\DeclareMathOperator{\SL}{{SL}}

\begin{document}

\title[]{Extensions and paracanonical curves of genus $6$}

\author[F. -O. Schreyer ]{Frank-Olaf Schreyer}
\address{	Mathematik und Informatik, Universit\"{a}t des Saarlandes, Campus E2 4,
	D-66123 Saarbr\"{u}cken, Germany}
\email{schreyer@math.uni-sb.de}

\author[H.L. Truong]{Hoang Le Truong}
\address{Institute of Mathematics, VAST, 18 Hoang Quoc Viet Road, 10307
Hanoi, Viet Nam}
\email{hltruong@math.ac.vn\\
	truonghoangle@gmail.com}


\begin{abstract} 

In this note we give a computationally easy to use method to compute a maximal extension of certain varieties. As a application we prove
that a general paracanonical curve $C$ genus $6$ as a codimension three subvarieties of  $\PP^4$ extend to precisely
  $26$ families of surfaces $Y \subset \PP^{5}$.

Moreover we give partial classification of maximally extended codimension three ACM subvarieties of degree $10$.

\end{abstract}

\maketitle


\section{Introduction}

This note grew out of the attempt to answer  Question 4.7 of \cite{AHK23} in a special case: Can we classify codimension $3$
subvarieties $Z$ of a projective spaces with Betti table
$$
\begin{matrix}\label{betTab}
          & 0 & 1 & 2 & 3\\ \hline
       0: & 1 & . & . & .\\
       1: & . & . & . & .\\
       2: & . & 10 & 15 & 6
       \end{matrix} \leqno{(1)}\label{betTab}
       $$
An examples of such a variety is defined by the $3\times 3$ of a generic $3\times 5$ matrix of linear forms, i.e., $Z$ is the secant variety 
of $\PP^{2}\times \PP^{4} \subset \PP^{14}$.
Not all are of this type. The curve section $C=Z \cap\PP^{4}\subset \PP^{14}$ for a general linear subspace $\PP^{4} \subset \PP^{14}$ 
is isomorphic to a smooth plane quintic, while a general paracanonical curve $C  \subset \PP^{4}$ of genus $6$ has the Betti table (\ref{betTab}) above
but is not isomorphic to a smooth plane quintic.

Our approach to this classification problem is to compute maximal extensions of paracanonical curves
using the deformation theory of the second and third matrix in the free resolution of coordinate ring of $C$.

We believe that the computationally simple approach of Theorem \ref{extensions} to extending varieties is of independent interest.

Our main result Theorem \ref{thm1} says the following: \medskip

{\it A general paracanonical curve $C \subset \PP^{4}$ of genus $6$ can be extended to precisely $26$ different families of surfaces $Y \subset \PP^{5}$,
while a general Prym canonical curve  $C \subset \PP^{4}$ can be extended to precisely $27$ different families of surfaces, where the extra family 
is the unique Fano polarized Enriques surface $Y \subset \PP^{5}$ which contains $C$ as a hyperplane section.} \medskip

In case of curves isomorphic to a general smooth plane quintic we have instead Theorem \ref{thm2}: \medskip

{\it A general paracanonical curves which is isomorphic to a smooth plane quintic has a single family of $1$-extension $Y$ 
all of which are determinantal surfaces. A general
Prym canonical curve $C \subset \PP^{4}$ which isomorphic to a smooth plane quinic has two families of $1$-extensions 
where the second family is the unique Enriques surface $Y$, which contains $C$ as a hyperplane section.} \medskip
 
Smooth surfaces $Y\subset \PP^{5}$ of degree $10$ with the Betti table as above have been classified in \cite{Truong}.
The families are characterized by the self-intersection number $K_{Y}^{2}$ which can take values in $\{-6, \ldots, 0 \}$, 
see also Table (\ref{table1}). We have observed that for a general member in each family the obstruction to extend a first order 
extension to a full extension, i.e.  the ideal $eq$ of Theorem \ref{extensions}, vanishes.  
However we do not have an explanation of this phanomen. Presumably, a theoretical proof would establish this for all $3$-regular (smooth?) 
surfaces $Y\subset \PP^{5}$ of degree $10$.

The paper is organized as follows. Section \ref{extend} contains our main result Theorem \ref{extensions} on the computation of extensions.
Section \ref{paraCurves} contains the proof of our two main Theorems.

In Section \ref{maxExtension} we comment on geometry of the maximal extensions of the general surfaces 
in this classification.  
In Section \ref{familiesOfPairs} we summarise the consequences  for the moduli spaces of pairs $(Y,C)$, and in Section \ref{specialCurves}
we report about our findings for special $C$.\medskip

We are mainly interested in case of curves and surfaces over the ground field $\CC$. However 
our computations documented in our Macaulay2 \cite{M2} package
SurfacesAndExtensions
\cite{STm2} are usually over a finite prime
field $\FF_{p}$, which we regard as the reduction mod $p$ of a computation of an example defined over an open part of $\Spec \ZZ$.
Semi-continuity will imply that the results are true for an example defined over $\QQ$ and hence over $\CC$. We leave the very interesting case 
of very small characteristics aside. \medskip

{\bf Acknowledgement.} We thank Simon Brandhorst, Thomas Dedieu, Igor Dolgachev, Claus Fieker and Alessandro Verra for valuable discussions.
This paper is a contribution to the  Project-ID 286237555 - TRR 195 of the Deutsche Forschungsgemein-
schaft (DFG, German Research Foundation) of the first author. H.L.Truong was partially supported by Tosio Kato Fellowship and the Vietnam Academy of Science and Technology (VAST) under grant number  CTTH00.03/24-25.

\section{Extensions}\label{extend}

Let $X \subset \PP^{n}$ be a variety. An extension $Y \subset \PP^{n+e}$ of $X$ is a variety such that $X=\PP^{n}\cap Y$
where $\PP^{n} \subset \PP^{n+e}$ is the span of the first $n+1$ coordinate points of $\PP^{n+e}$, such that the defining linear forms
$x_{n+1},\ldots,x_{n+e}$ of $\PP^{n} \subset \PP^{n+e}$ form a regular sequence on the homogeneous coordinate ring $S_{Y}$ of $Y$ and 
$S_{X}=S_{Y}/(x_{n+1},\ldots,x_{n+e})S_{Y}$. 

Thus if $X$ is an arithmetically Cohen-Macaulay (ACM) variety, then $Y$ is ACM as well.
If $X$ and $Y$ are not cones  then we speak of a non-trivial extension.

\begin{thm}\label{extensions} Let $X \subset \PP^{n}$ be a variety of codimension $c$ which is not a cone, whose minimal free resolution $F$
$$
\xymatrix{
0 & \ar[l] S_{X}  &\ar[l] F_{0} & \ar[l] \ldots & \ar[l] F_{i-1} & \ar[l]_{\varphi_{i}} F_{i} &\ar[l]_{\varphi_{i+1}} F_{i+1} & \ar[l] \ldots &\ar[l] F_{pd}& \ar[l] 0 \\
}
$$
has two consecutive differentials $\varphi_{i}, \varphi_{i+1}$ with $i+1\le c$ given by linear matrices. Consider
the linear equations obtained from
$$\tilde A\cdot \varphi_{i+1}+\varphi_{i}\cdot \tilde B=0$$
for the entries  of the $\rank F_{i-1}\times \rank F_{i}$ matrix $\tilde A=(a_{ij})$ and  $\rank F_{i}\times \rank F_{i+1}$ matrix $\tilde B=(b_{jk})$
by comparing coefficients. 
Suppose 
$$\PP^{m}\subset \PP^N \hbox{ with } N=\rank F_{i-1}\cdot\rank F_{i}+\rank F_{i}\cdot \rank F_{i+1}-1$$
is the linear solution space. Let
$A$ and $B$ be the linear matrices parametrized by $\PP^{m}$. Then maximal non-trivial linear extension of $X$ are in bijection with maximal linear 
subspaces of the vanishing loci $V(eq) \subset \PP^{m}$ where $eq$ denote the ideal generated by the 
entries of the $\rank F_{i-1}\times \rank F_{i+1}$ matrix $AB$. 
\end{thm}

\begin{proof} Suppose $X \subset Y \subset \PP^{n+e}$ is an extension. We use coordinates 
$x_{0},\ldots,x_{n}, y_{1},\ldots,y_{e}$ on $\PP^{n+e}$ such that $V(y_{1},\ldots,y_{e})=\PP^{n}$.
The minimal free resolution $F$ is obtained from the minimal free resolution $G$ by restriction:
$F=G/(y_{1},\ldots,y_{e})G$ as an $K[x_{0},\ldots,x_{n}]=K[x_{0},\ldots,x_{n},y_{1},\ldots,y_{e}]/(y_{1},\dots,y_{e})$-module, 
since
 $y_{1},\ldots, y_{e}$ a regular sequence on the homogeneous coordinate ring of $S_{Y}$ by the definition of an extension.
Writing $\Phi_{i}(x,y)$ and $\Phi_{i+1}(x,y)$ for the $i$-th and $(i+1)$-st differential of $G$ we have
$\varphi_{i}=\Phi_{i}(x,0)$ and $\varphi_{i+1}=\Phi_{i+1}(x,0)$.
Since $\Phi_{i}\cdot \Phi_{i+1}=0$ we see that $A'=\Phi_{i}(0,y)$ and $B'=\Phi_{i+1}(0,y)$ satisfy of the system of equations
$$A'\cdot \varphi_{i+1}+\varphi_{i}\cdot B'=0 \hbox{ and } A'\cdot B'=0.$$
Actually  $A=\Phi_{i}$ and $B=\Phi_{i+1}$ is a solution as well.

Conversely, given a maximal linear subspace $\PP^{n+e}$ of $V(eq) \subset \PP^{m} \subset \PP^{N}$. Then $\PP^{n}$ can be identified with 
a subspace $\PP^{n+e}$ because the trivial solution $A''=\varphi_{i}$ and $B''=\varphi_{i+1}$ corresponds to a subspace. Write
$$A|_{\PP^{n+e}}=A''+A' \hbox{ and } B|_{\PP^{n+e}}=B''+B'.$$
Consider the three term complex 
$$
\xymatrix{ G_{i-1}& \ar[l]_{A''+A'} G_{i} & \ar[l]_{B''+B'} G_{i+1}.\\}
$$
It can be viewed as the presentation a deformation of the module $\coker(\varphi_{i})$ with base $\mathbb A^{e}$ whose syzygies lift. 
By \cite{Artin} the higher syzygies lift as well which give us a complex
$$
\xymatrix{ G_{i-1}& \ar[l]_{A''+A'} G_{i} & \ar[l]_{B''+B'} G_{i+1}& \ar[l] G_{i+2}
& \ar[l] \ldots & \ar[l] G_{c}& \ar[l] 0.}
$$
Like wise 
$$
\xymatrix{ G_{i-1}^{*} \ar[r]^{(A''+A')^{t}} & G_{i}^{*}  \ar[r]^{(B''+B')^{t}} & G_{i+1}^{*}\\}
$$
can be viewed as a deformation of $\coker(\varphi_{i+1}^{t})$ with a lifting of sygygies. Thus this maps
extends to a deformation of the complex
$$
\xymatrix{
0\ar[r] & F_{0}^{*} \ar[r]& \ldots  \ar[r] &F_{i-1}^{*}  \ar[r]^{\varphi_{i}^{t}} & F_{i}^{*} \ar[r]^{\varphi_{i+1}^{t}}& F_{i+1}^{*} 
}
$$
which is exact since $\Ext^{\ell}_{S}(S_{X},S)=0$ for $\ell <c$.
We obtain the homogeneous coordinate ring $Y$ as the cokernel of the dual of the map $G_{0}^{*} \to G_{1}^{*}$. 
\end{proof}

\begin{ex}[Pinkham's example \cite{Pinkham}] Let $C \subset \PP^{4}$ denote the rational normal curve of degree $4$. 
The ideal $I_{C} \subset S=K[x_{0},\ldots,x_{4}]$ of
$C$ is defined by the $2\times 2$-minors of the matrix 
$\begin{pmatrix}
x_{0} & x_{1} & x_{2} & x_{3} \cr
x_{1} & x_{2} & x_{3} & x_{4} 
\end{pmatrix}$
and  its coordinate ring $S_{C} =S/I_{C}$ is resolved by the Eagon-Northcott complex
$$
\xymatrix{ 0 & \ar[l] S_{C} &  \ar[l] S &  \ar[l] S^{6}(-2) & \ar[l]_{\varphi_{2}} S^{8}(-3) & \ar[l]_{\varphi_{3}} S^{3}(-4) & \ar[l] 0. \\}
$$
In this case the linear matrices $A$ and $B$ depend on $9$ parameters and the ideal $eq$ decomposes
 $$eq=\left(b_{5,2}b_{6,1},\,b_{4,1}b_{5,2}-b_{5,2}b_{6,2},\,b_{1,1}b_{5,2}-
       b_{4,2}b_{5,2}\right)=\left(b_{5,2}\right)\cap \left(b_{6,1},\,b_{4,1}-b_{6,2},\,b_{1,1}-b_{4,2}\right)$$
into two linear components. The first correspond to the extension to $\PP^{1}\times \PP^{3} \subset \PP^{7}$
the second to the Veronese surface $V \subset \PP^{5}$.
\end{ex}

\begin{ex}[Tom and Jerry, \cite{Reid}] Consider the del Pezzo surface $X$ of degree $6$ in $\PP^{6}$. It has syzygies
$$
\xymatrix{ 0  & \ar[l] S_{X} &  \ar[l] S &  \ar[l] S^{9}(-2) & \ar[l]_{\varphi_{2}} S^{16}(-3) & \ar[l]_{\varphi_{3}} S^{9}(-4) &   \ar[l] S(-6) &\ar[l] 0.}
$$
The $9\times 16$ and $16\times 9$ matrices $A$ and $B$ depend on $10$ parameters and the ideal $eq$ decomposes into
$$eq=\left(b_{13,7}b_{15,8}+b_{14,8}b_{15,8}-b_{15,8}^{2},\,b_{11,7}b_{13,7}+b_{11,7}b_{14,8}-b_{11,7}b_{15,8}\right)= \left(b_{13,7}+b_{14,8}-b_{15,8}\right) \cap \left(b_{15,8},\,b_{11,7}\right) 
$$
which correspond to the extensions to $\PP^{2}\times \PP^{2} \subset \PP^{8}$ and $\PP^{1}\times \PP^{1} \times \PP^{1} \subset \PP^{7}$
respectively.
\end{ex}
Both example are computed with our Macaulay2 package 
SurfacesAndExtensions \cite{STm2}.

\section{Extensions of  general paracanonical curves of genus $6$}\label{paraCurves}

Let $C$ be a smooth projective curve of genus $6$. For a non-trivial line bundle $\eta \in \Pic^0(C)$, we shall study the paracanonical line bundle $L := K_C \otimes \eta$. When $\eta \in \Pic^{0}(C)[2]$ is a non-trivial point, then $L$ is a Prym-canonical line bundle. For each paracanonical bundle $L$, we have $h^0(C,L) = 5$ and an embedding
$$ \Phi_L: C \hookrightarrow \Bbb P^4$$
unless $\eta \in C^{2}- C^{2}$ lies in the second difference variety of $C$. In the following we assume that $\eta \notin  C^{2}- C^{2}$
and call $C=\Phi_L(C)\subset \PP^{4}$ a paracanonical embedded curve.

Since $h^{0}(\PP^{4},\sO(2))=15=h^{0}(C,L^{\otimes 2})$ 
pairs $(C,\eta)$ such that the coordinate ring of $\Phi_{L}(C)$ is not arithmetically Cohen-Macaulay form a divisor in the moduli space $\fkP ic^{0} _{6}\to \fkM _{6}$, see \cite{Farkas12}.
In the following we consider only pairs $(C,\eta)$ outside this divisor, i.e., with $H^{1}(\PP^{4},\sI_{C}(2))=0$.

\begin{prop} Let $C \subset \PP^{4}$ be a paracanonical embedded curve of genus $6$ with $H^{1}(\PP^{4},\sI_{C}(2))=0$. Then $C$ has Betti table
$$\begin{matrix}
          & 0 & 1 & 2 & 3\\ \hline
       0: & 1 & . & . & .\\
       1: & . & . & . & .\\
       2: & . & 10 & 15 & 6 
       \end{matrix}$$
\end{prop}

\begin{proof} Let $V \subset H^{0}(C,L)$ be a two dimensional subspace corresponding to a base point free pencil.
Following \cite{Mumford}  we obtain a commutative diagram
$$
\xymatrix{
 && V\tensor S_{n}H^{0}(C,L)\ar[d] \ar[r]& S_{n+1} H^{0}(C,L) \ar[d] &\cr
 0 \ar[r] &H^{0}(C,L^{\tensor (n-1)}) \ar[r]& V\tensor H^{0}(C,L^{\tensor n}) \ar[r]& H^{0}(C,L^{\tensor(n+1)})\ar[r] &0\cr
}
$$
with the right map in bottom row surjective for $n\ge 2$ since $(n+1)d+1-g=2(nd+1-g)-((n-1)d+1-g)$. 
Thus the surjectivity $S_{n}H^{0}(C,L) \to H^{0}(C,L^{\tensor n})$ for $n=2$ implies the surjectivity for all $n\ge 3$ by induction and the commutativity
of the diagrams. Hence
$C \subset \PP^{4}$ is  arithmetically Cohen-Macaulay with artinian Hilbert function $(1,3,6,0,\ldots)$. The result follows because
the Hilbert series of the homogeneous coordinate ring $S_{C}$ of $C \subset \PP^{4}$ satisfies
$$H_{S_{C}}(t)= \frac{1-10t^{3}+15t^{4}-6t^{5}}{(1-t)^{5}}=\frac{1+3t+6t^{2}}{(1-t)^{2}}.$$
\end{proof}

 \begin{thm}\label{thm1}
Let $C \subset \PP^{4}$ be a general  smooth paracanonical curve  of genus $6$. Then 
\begin{enumerate}[$1)$]
\item there is a unique family of  $1$-extensions $Y$ of $C$ with $K_Y^2=-5$,
\item there are $5$ families of  $1$-extensions $Y$ of $C$ with $K_Y^2=-4$,
\item\label{3)} there are $20$ $1$-extensions $Y$ of $C$ with $K_Y^2=-3$,
\end{enumerate}
and no other.
If $C$ is a general Prym canonically embedded curve then there is apart from  26 families from above one more surface,
\begin{enumerate}[$4)$]
\item there is a unique $1$-extension $Y$ of $C$ which is a Fano polarized Enriques $Y$, hence $K_{Y}^{2}=0$.
\end{enumerate}
\end{thm}

\begin{proof}
Let $C \subset Y\subset \PP^{5}$ be an extension. Then $Y$ has the same Betti numbers as $C$. In particular $Y$ is an arithmetically Cohen-Macaulay variety,
$H^{1}(Y,\sO_{Y})=H^{1}(Y,\omega_{Y})=0$ and $H^{0}(Y, \omega_{Y})=0$.
By the adjunction sequence
$$ 0 \to \omega_{Y} \to \omega_{Y}(C) \to \omega_{C} \to 0$$
the canonical system on $C$ is cut out by the linear system $|K_{Y}+C|$,
and the canonical curve of $C$ is contained in the image $Y_{1}$ of $$\varphi_{|K_{Y}+H|}\colon Y\to \PP^{5}.$$
By \cite{SVdV87} $\varphi_{|K_{Y}+H|}\colon Y \to Y_{{1}}$ contracts all $(-1$)-lines, i.e., all $(-1)$ curves $E \subset Y \subset \PP^{5}$ 
which are embedded in $\PP^{5}$ as lines. Outside these lines $\varphi_{|K_{Y}+H|}$ is bi-regular. The rational map $Y_{1} \dasharrow Y$ is 
defined by the linear subsystem of $|C|$ on $Y_{1}$ with base points at the images of the exceptional lines.

The Betti table of the canonical image of $C$ decomposes
\begin{center}
\begin{tabular}{|ccccc|}
\hline
     1&\text{.}&\text{.}&\text{.}&\text{.}\\
     \text{.}&6&5&\text{.}&\text{.}\\
     \text{.}&\text{.}&5&6&\text{.}\\
     \text{.}&\text{.}&\text{.}&\text{.}&1\\ \hline
  \end{tabular} =
  \begin{tabular}{|cccc|}
  \hline
     1&\text{.}&\text{.}&\text{.}\\
     \text{.}&5&5&\text{.}\\
     \text{.}&\text{.}&\text{.}&1\\ \hline
  \end{tabular} $\otimes$
  \begin{tabular}{|cc|}
      \hline
      1&\text{.}\\
      \text{.}&1 \\ \hline
     \end{tabular}
\end{center}
indicating that the canonical curve $C \subset \PP^{5}$ is a complete intersection of a del Pezzo surface $D$ of degree $5$ and a quadric, see e.g. \cite{Sch86}.
Recall that over an algebraically closed field $D$ is isomorphic to $\PP^{2}$ blown up in $4$ points embedded by $|3L-\sum_{i=1}^{4} E_{i}|$,
where $L$ denotes the class of a general line and the $E_{i}$ denote the exceptional divisors.
Thus $C$ has a plane model of degree $6$ with $4$ double points. The curve $C$ has five $g^{1}_{4}$ of which four correspond  to a projection 
from a double point. The fifth is cut out on $C$ by the pencil of conics through the 4 double points.

\begin{enumerate}[$1)$]
\item The first case occurs if $Y_{1}=\varphi_{|K_{Y}+H|}(Y)$ coincides with this del Pezzo surface $D$. In that case $Y$ is the projection of $D$ re-embeddd with
$|6L-\sum_{i=1}^{4}2E_{i}|$ from ten points, to obtain the desired degree
$$\deg Y=10=6^{2}-4\cdot 2^{2}-10.$$
Thus $Y$ is isomorphic to $\PP^{2}$ blow-up in $14$ points, embedded by
$|6L-\sum_{i=1}^{4}2E_{i}-\sum_{j=5}^{14}E_{j}|$ and $K_{Y}^{2}=9-14=-5$.

\item In the second case $Y_{1}=\varphi_{|K_{Y}+H|}(Y)$ is a complete intersection of $\PP^{1}\times \PP^{2} \subset \PP^{5}$ with a quadric.
For each general curve $C$ we obtain 5 families of such surfaces as follows.
Consider the map 
$$
\varphi_{|D|\times |K_{C}-D|}\colon C \to \PP^{1}\times \PP^{2}\subset \PP^{5}
$$ 
where $|D|$ is one of the five $g^{1}_{4}$'s and
$|K_{C}-D|$ is the Brill-Noether dual $g^{2}_{6}$. The image of $\PP^{1}\times \PP^{2}$ in $\PP^{5}$  is defined the $2\times 2$ minors of the 
$2\times 3$ matrix of linear forms on $\PP^{5}$ obtained form
$$H^{0}(C,\sO(D))\times H^{0}(C,\omega_{C}(-D)) \to H^{0}(C,\omega_{C})\cong H^{0}(\PP^{5},\sO(1)).$$
Then $Y_{1}$ is the complete intersection of $\PP^{1}\times \PP^{2}$ and a quadric $Q$ defined by a further degree two generator of 
$I_{C}$. The surface $Y_{1}$
is a conic bundle over $\PP^{1}$ with six singular fibers. So $K_{Y_{1}}^{2}=8-6=2$. Since 
$$6=\deg Y_{1} =(C+K_{Y}).(C+K_{Y})=2(2\cdot 6-2)-C^{2}+K_{Y}^{2}=10+K_{y}^{2}$$ 
we have $K_{Y}^{2}=-4$. The surface $Y$ is obtained from $Y_{1}$ by blowing-up $\ell=K_{Y_{1}}^{2}-K_{Y}^{2}=2+4=6$ points. 

\item In the last case, $Y_{1} \subset \PP^{5}$ is linked via a complete intersection of three quadrics to a plane. Thus $Y_{1}$ has Betti table
$$\begin{matrix}
         & 0 & 1 & 2 & 3\\
         \hline
      0: & 1 & . & . & .\\
      1: & . & 3 & . & .\\
      2: & . & 1 & 6 & 3
      \end{matrix}$$
  By the adjunction process $Y_{1}$ is isomorphic to the blow-up of $\PP^{2}$ in $\ell$ points re-embedded by the linear system 
  $|4L-\sum_{j=1}^{\ell}E_{j}|$. So
  $7=\deg Y_{1}=16-\ell$ implies $\ell=9$ and hence $K_{Y_{1}}^{2}=0$.  On the other hand,
 since $7 =\deg Y_{1}=(C+K_{Y}).(C+K_{Y})=2(2\cdot 6-2 )-10+K_{Y}^{2}$ we have $K_{Y}^{2}=-3$ and hence $\varphi_{|C+K_{Y}|}\colon Y \to Y_{1}$ 
 blows down precisely $3$ points. So $Y$ is isomorphic to the blow-up of $\PP^{2}$ in $12$ points re-embedded by 
 $|7L-\sum_{j=1}^{9} 2E_{j}-\sum_{i=10}^{12} E_{i}|$.
 Such surfaces $Y$ arises form a given general paracanonical curve $C \subset \PP^{4}$ as follows. Let $\eta \in \Pic^{0}(C)$ be the bundle such that
$\sO_{C}(1)= \omega_{C} \otimes \eta$.
The difference map 
$$\nu_{3}\colon C^{(3)}\times C^{(3)} \to \Pic^{0}(C), \, (D,E) \mapsto \sO_{C}(D-E)$$
is surjective and finite to one of degree ${6 \choose 3}=20$ by \cite[Exercise D-3]{ACGH85}.
Pick $D=q_{1}+q_{2}+q_{3}$ and $E=p_{1}+p_{2}+p_{3}$ whose difference map to $\eta$ and consider the canonical embedding 
$C \hookrightarrow \PP^{5}$. Let $P=\langle q_{1},q_{2},q_{3} \rangle \subset \PP^{5}$ be the plane spanned by $\{q_{1},q_{2},q_{3}\}.$
There exists a $3$-dimensional subspace $V \subset H^{0}(\PP^{5},\sI_{C}(2))$ of quadrics containing the plane $P$:
Each quadric in $H^{0}(\PP^{5},\sI_{C}(2))$ vanishes already in $\{q_{1},q_{2},q_{3}\}$. Hence the condition to vanish on $P$ 
imposes only $3 =h^{0}(\PP^{2},\sO(2))-3$ further linear conditions on the $6$-dimensional space $H^{0}(\PP^{5},\sI_{C}(2))$.
For general choices of $C$ and $D$ these three quadrics form a complete intersection whose zero loci decomposes
as $T_{7}  \cup P$, which $T_{7}$ the image of a $\PP^{2}$ blown-up in $9$ points embedded by
$$
|4L- \sum_{j=1}^{9} E_{j}|
$$
as above. Since $C \subset T_{7}$ the points $\{p_{1},\ldots,p_{3}\}$ give further three points on $T_{7}$.
The image $Y$ of $\PP^{2}$ blown-up in $12$ points embedded by $|7L-\sum_{j=1}^{9}2E_{j}-\sum_{i=10}^{12}E_{i}|$ is our desired surface $Y$
whose first adjoint surface $Y_{1}$ coincides with $T_{7}$. The surfaces $P$ and $T_{7}$ intersect in a plane cubic curve, whose class on $T_{7}$ is $3L-\sum_{i=1}^{9} E_{i}$. Since $(3L-\sum_{i=1}^{9} E_{i}).(7L- \sum_{i=1}^{9} 2E_{i})=3$
this cubic intersects $C$ precisely in the points $\{q_{1},q_{2},q_{3}\}$ and
$\eta=\sO_{Y}(-K_{Y})\otimes \sO_{C} \cong \sO_{C}(D-E)$.
  
 \end{enumerate}
 
 This proves that a general paracanonical curve $C \subset \PP^{5}$ has at least 26 families of extensions. To prove that equality holds it suffices 
 to establish that 
 there no other extensions in a single example over a finite field by semi-continuity. We check this using the technique from Theorem \ref{extensions}.
 In our example the vanishes loci $V(eq)$ decomposes in $26$ linear spaces of dimension $9$, $7$ and $5$ respectively, as running the 
 code printed out by the function extensionsOfAGeneralParacanonicalCurve in our Macaulay2 package
 SurfacesAndExtensions \cite{STm2} shows. In particular the collection of the twenty surfaces $Y$ in \ref{3)}) are uniquely determined by $C$.
 
 If $C$ is a general Prym canonical curve, then by \cite{Verra} $C$ lies on an Enriques surface. So we have
 \begin{enumerate}[$4)$]
 \item $C$ is the hyperplane section of a Fano polarized Enriques surface. 
 \end{enumerate}
 Thus in the Prym case there are at 
 least $27$ families of extensions. Running the Macaulay2 code printed by our function extensionsOfAGeneralPrymCanonicalCurve  of our package
 SurfacesAndExtensions \cite{STm2} shows that there are no further families.
\end{proof}

 \begin{thm}\label{thm2} Let $C\subset \PP^{4}$ be a general paracanonical curve which is isomorphic to smooth plane quintic. Then $C$ has  a
 single family of $1$-extensions $Y\subset \PP^{5}$. All $Y$'s are $\PP^{5}$ sections of the generic determinantal variety in $\PP^{14}$. 
 For a general Prym canonical curve $C$, which is isomorphic to a smooth plane quintic, there are two families of $1$-extensions.
 Apart from the determinantal family there is the unique Fano polarized Enriques surface $Y$ which contains $C$ as a hyperplane section.
 \end{thm}
 
 \begin{proof} A projectively normal paracanonical curve $C$ which is isomorphic to a smooth plane quintic is determinantal. Indeed, consider $C \subset \PP^{2}$ 
 the plane model and an effective divisor of $D$ of degree $10$ with $\sO_{C}(D)\cong \omega_{C}\otimes \eta$. 
 If the scheme $D \subset C \subset \PP^{2}$ is contained in a cubic $F$, then $|D|$ is cut out by the cubics through the residual $5$ points of  $C\cap F$ to $D$.
 In particular the image of $C$ is contained in a (possibly singular) del Pezzo surface of degree $4=3^{2}-5$ and hence in two quadrics. 
 Hence $C \subset \PP^{4}$ is not projectively normal, a contradiction. So $D\subset \PP^{2}$ is not contained in a cubic. We conclude that $D$ is contained in ${4+2 \choose 2}-10=5$ quartics
 which are the minors of a $5\times 4$ linear Hilbert-Burch matrix \cite[Chapter 20]{Eisenbud} of $D$. The inclusion $D \subset C$, says that the equation of $C$ is a linear combination
 of these minors. In other words, $C\subset \PP^{2}$ is defined by a $5\times 5 $ matrix $\varphi$ of linear forms. Interpreting the triple tensor
 $$\phi \in H^{0}(\PP^{2},\sO(1)) \tensor H^{0}(C,L) \tensor V$$
 as a $3\times \dim V$ matrix $m_{3\times 5}$ of linear forms in $H^{0}(C,L)=H^{0}(\PP^{4}, \sO(1))$ we see that $C \subset \PP^{4}$ is defined by the $3\times 3$-minors of
 $m_{3\times 5}$.
 
 Note that the transpose matrix $\varphi^{t}$ leads to the paracanonical embedding by $\omega_{C}\tensor \eta^{{-1}}$. In particular we see that for a projectively normal Prym canonical curve, the matrix $\varphi$ above can be chosen to be symmetric.
 Since $C\subset \PP^{4}$ is determinantal we have the family of determinantal $1$-extensions. In the symmetric case there is in addition the $1$-extension to the Fano polarized Enriques
  surface. In the documentation of our functions  paracanonicalEmbeddedPlaneQuintic and
prymCanonicalEmbeddedPlaneQuintic of our package SurfacesAndExtensions \cite{STm2} shows that for general $C$ there are no further extension.
  \end{proof}

\section{Maximal extensions}\label{maxExtension}

Smooth surfaces $Y \subset \PP^{5}$ with Betti table (\ref{betTab}) have been classified in \cite{Truong}. The main result is summarized in Table \ref{table1}. A result which surprised us, is that the obstruction ideal $eq$ for extending a general surface in this list is always zero. We prove this by computing the extension of examples over a finite field, see the documentation of our function 
threeRegularSurfacesOfCodim3 of out package SurfacesAndExtensions \cite{STm2} and apply semi-continuity.

\begin{small}
\begin{table}[h!]
\centering
		\begin{tabular}{ |r| c| c| c|c|c|c|}
	\hline
	  $K^2_Y$ &  structure of $Y$ & Betti table of $Y_{1}$ &$\deg Y_{1}$ &  curve section &$n+e$ &max. extension $Z$\\
	  	\hline
	$-6$  & $\Bbb P^2(5;1^{15})$  &
	$\begin{matrix}
	0: & 1 & . & . & .\\
	1: & . & 6 & 8 & 3
	\end{matrix}$ 
	& 4 &
	\begin{tabular}{c}
	 paracanonical \\
	 model of a \\
	 plane quintic
	 \end{tabular} 		
	 &  $14$ &\begin{tabular}{c}
	 $3\times 5$
	determinantal\\
	secant  variety of\\
	$\PP^{2}\times \PP^{4} \subset \PP^{14}$
	\end{tabular} \\ 
\hline
	$-5$   & $\Bbb P^2(6;2^4,1^{10})$  & 
	$\begin{matrix}
   	0: & 1 & . & . & .\\
   	1: & . & 5 & 5 & .\\
   	2: & . & . & . & 1
      \end{matrix}$ 
      & $5$&
      \begin{tabular}{c}
	  general \\
	 paracanonical \\
	 curve 
	 \end{tabular} 	
      &$9$& 
      \begin{tabular}{c}
	  base locus of a
	 cubic\\-quadric Cremona\\
	 transformation of $\PP^{9}$
	 \end{tabular} 
      \\ 
\hline
	$-4$   & $\Bbb P^2(7;3^1,2^6,1^{6})$  & 
	$\begin{matrix}
	0: & 1 & . & . & .\\
	1: & . & 4 & 2 & .\\
	2: & . & . & 3 & 2
	\end{matrix}$
	& $6$ &
	\begin{tabular}{c}
	  general \\
	 paracanonical \\
	 curve 
	 \end{tabular} 	
	& $7$& 
	\begin{tabular}{c}
	  projection of the image of\\ 
	  the $\PP^{1}$-bundle 
	  $\PP(\sF)$ in $\PP^{13}$ \\
	  from $6$ points  where $\sF$ is\\
	  $\sO_{\PP^{1}\times \PP2}(1,0)\oplus \sO_{\PP^{1}\times \PP^{2}}(1,2))$\\		 
	 \end{tabular} 	\\ 
\hline
	$-3$  & $\Bbb P^2(7;2^9,1^{3})$  & 
	$\begin{matrix}
	0: & 1 & . & . & .\\
	1: & . & 3 & . & .\\
	2: & . & 1 & 6 & 3
\end{matrix}$
	&$7$&
	\begin{tabular}{c}
	  general \\
	 paracanonical \\
	 curve 
	 \end{tabular} 	
	 & $5$ &\% \\ 
\hline
	$-2$   & $\PP^2(9;3^6,2^4,1^{1})$  &
	$\begin{matrix}
	0: & 1 & . & . & .\\
	1: & . & 2 & . & .\\
	2: & . & 4 & 9 & 4
\end{matrix}$	
	 &$8$ &
	 \begin{tabular}{c}
	special pair \\ 
	$(C,\eta)$	
	\end{tabular}
 &6 & \begin{tabular}{c}
	  blow-up of a toric\\
	  complete intersection\\
	  of two quadrics in $\PP^{5}$
	 \end{tabular} 	\\ 
\hline	
	$-1$   & $\Bbb P^2(10;3^{10})$  & 
	$\begin{matrix}
	0: & 1 & . & . & .\\
	1: & . & 1 & . & .\\
	2: & . & 7 & 12 & 5
	\end{matrix}$	
	&$9$&
	\begin{tabular}{c}
	special pair \\ 
	$(C,\eta)$	
	\end{tabular}
	 & $5$& \% \\ 
\hline
	$0$  & 
	\begin{tabular}{c}
	Fano model of an\\ 
	Enrique surface 
	\end{tabular}&  
	$\begin{matrix}
	0: &1 & . & . & .\\
	1:&. & . & . & .\\
	2: & . & 10 & 15 & 6
	\end{matrix}$	
	&10&
	\begin{tabular}{c}
	  Prym \\
	 canonical curve 
	 \end{tabular} 	
	 &
	 $5$& \% \\ 
\hline
	$0$  & 
	 \begin{tabular}{c} Fano model of\\
	 a nodal\\
	 Enriques surface  
	 \end{tabular}
	 &
	$\begin{matrix}
       1 & . & . & . & . & .\\
       . & 1 & . & . & . & .\\
        . & 4 & . & . & . & .\\
        . & . & 14 & 15 & 6 & 1
        \end{matrix}$
	   &$10$&
	   \begin{tabular}{c}
	  special Prym \\
	  canonical curve 
	  \end{tabular} 	   
	   &9&\begin{tabular}{c} secant variety of\\
	   the Veronese 3-fold \\$v_{2}(\PP^{3})\cong \overline{SO(3)} \subset \PP^{9}$
	 \end{tabular}\\ 
\hline	
\end{tabular}\vspace{0.2cm}

\caption{ Smooth arithmetically CM surfaces $Y \subset \PP^{5}$  of degree $10$ over an algebraically closed field of characteristic zero, their first adjoint $Y_{1} \subset \PP^{5}$
and comments on their maximal extension.}
\label{table1}
\end{table}

\end{small}
The description of the maximal extensions uses  computations as well. Details can be found by running the code which our function
identifyMaximalExtensions provides. A summary of our findings is the following:

\begin{enumerate}
\item[-6)]  The $2\times 2$ minors of the generic $3\times 5$ matrix define the Segre product of $\PP^{2}\times \PP^{4} \subset \PP^{14}$. Hence the $3\times 3$ minors define its secant variety.

\item[-5)] The best way to describe the maximal extension $Z$ is via a Cremona transformation. The ten cubic equations of $Z$
define a birational map $\sigma\colon  \PP^{9} \dasharrow \PP^{9}$ whose exceptional divisor is a quintic, which is the annihilator
 of  the cokernel of the $10\times 10$ jacobian matrix of the equations. The inverse Cremona transformation $\tau$ is defined by quadrics
 which define a $4$-fold base locus $W\subset \PP^{9}$. The easiest way to get the pair $\sigma, \tau$ is via $W$. The variety $W$
 is the image of the rational map $\PP^{4} \dasharrow \PP^{9}$ defined by the $10$ quadrics which generate the homogeneous ideal of the $5$ coordinate points of $\PP^{4}$.

\item[-4)] The adjoint variety of $Z$ is $\PP^{1}\times \PP^{2}\subset \PP^{5}$. The rational map $Z \dasharrow \PP^{1}\times \PP^{2}$ has generically fibers isomorphic to $\PP^{1}$. There are precisely $6$ fibers which are not isomorphic to $\PP^{1}$. Those fibers are $\PP^{3}$'s.  

The variety $Z$ arises as follows. Consider the  $\PP^{1}$-bundle $\PP(\sF)$ over $\PP^{1}\times \PP^{2}$ for the split bundle $$\sF=\sO_{\PP^{1}\times \PP^{2}}(1,0)\oplus \sO_{\PP^{1}\times \PP^{2}}(1,2).$$
The morphism $\varphi_{|\sO_{\sF}(1)|}\colon \PP(\sF)\to Z' \subset \PP^{13}$ contracts the section corresponding to the first summand to a line and is bi-regular otherwise. $Z$ is the projection of $Z'$ from six points.
The projection collapses the six fibers of $\PP(\sF) \to \PP^{1} \times \PP^{2}$ containing a projection point to six isolated singularities, otherwise the projection is bi-regular.
The singular loci of $Z$ consists of the projected line and the six isolated points. Up to automorphism of $\PP(\sF)$ and $\PGL(8)$ the $Z$'s form a family of dimension
$6\cdot 4-(3+8+{2+2 \choose 2}+1)=6$ and taking sections with a $\PP^{5}$ leads to a $18$-dimensional family of surfaces. This coincides with number parameters $13\cdot 2-8$ for surface of type
$\PP(7;3,6^{2},6^{1})$.

\item[-2)]  The $3$-fold $Z$ is a blow-up of a toric complete intersection $Z'$ of two quadrics in $\PP^{5}$. More precisely, in suitable coordinates the complete intersection $Z'$ is defined by the two quadrics $$x_{1}x_{4}-x_{2}x_{5},\,x_{0}x_{3}-x_{2}x_{5}.$$
The complete intersection $Z'$ contains $8$ $\PP^{2}$'s. They correspond to the colomns of the 
skew symmetric matrices
$$\left(\begin{matrix}
       0&x_{1}&x_{2}&x_{3}\\
       -x_{1}&0&x_{0}&x_{5}\\
       -x_{2}&-x_{0}&0&x_{4}\\
       -x_{3}&-x_{5}&-x_{4}&0
       \end{matrix}\right)   \hbox{ and } 
\left(\begin{matrix}
       0&-x_{4}&x_{5}&-x_{0}\\
       x_{4}&0&-x_{3}&x_{2}\\
       -x_{5}&x_{3}&0&-x_{1}\\
       x_{0}&-x_{2}&x_{1}&0
       \end{matrix}\right)$$
The rational map $Z' \dasharrow Z$ blows up a  lines in each of the four $\PP^{2}$'s corresponding to the first matrix and the point $(1:1:1:1:1:1)$. The rational map is defined by seven cubics.  It collapse the strict transforms of these
planes to singular points, which are analytically isomorphic to a cone over the Veronese surface. There are 4 different descriptions of the rational map. Each of them
has four codimension one base loci.  They are $\PP^{2}$'s corresponding  to one of the columns of the first matrix and three colomns of the second matrix.

For more details we refer to the function identifyMaximalExtension of our package SurfacesAndExtensions \cite{STm2}. Up to projectivities, we have a $8=4\cdot 2$-dimensional family of maximal extended varieties $Z$ of this type, which give rise to
$14=8+6$-dimensional family of surfaces. This coincides with the dimension $2\cdot11-8$ of the family  of surfaces of type $\PP^2(9;3^6,2^4,1^{1})$.

\item[0)] The maximal extension $Z$ of a nodal Enriques surface  has as adjoint variety the Pl\"ucker quadric  $Q\subset \PP^{5}$.
$Z$ is birational to the projective bundle of a rank $3$ vector bundle $\sF$ obtained as the second syzygy sheaf of 15 of the 16 sections
of four copies of the spinor bundle $\sS$ on $Q$. Over the coordinate ring  $S_{Q}$ of $Q$ the bundle   $\sF$ sits in an  infinite  
complex with Betti numbers
$$
\begin{tabular}{cccc cccccc}
 16& 16 &15 & 10 & .   & .      & .   &  .  &  .  & \\
 .   &   .  &  .  &  .   & 10 & 15 & 16 &16 &16 &\ldots
\end{tabular}
$$
Up to symmetries of the total space of four copies of the spinor bundle  on $Q=\GG(2,4)$ there is only one choice of such a morphism $\sO_{Q}^{15} \to \sS^{4}$ and a direct computation shows that $Z$ can be identified with the secant variety of the Veronese $3$-fold $v_{2}(\PP^{3})\subset \PP^{9}$. 

Also in this case the equations of $Z$ define a Cremona transformation. This time it a self-dual Cremona transformation defined by the partials of the quartic, which defines $\Sec^{3}(v_{2}(\PP^{3}))\subset \PP^{9}$, c.f. \cite{EinShepherd-Barron}.

The nodal character of the general surface section $Y$ is the following: There are $20$ half-pencils, i.e. plane cubic curves $C_{i}$ in $Y$. They come in pairs $C_{i}$, $C_{i}'$ of numerically equivalent curves, which differ by the canonical class. 
Suitable labeled we have that the hyperplane $H$ class of $Y$ satisfies
$$H\sim \frac{1}{3}( C_{1}+\ldots +C_{10}).$$
These curves are called half-pencils, because $2C_{i} \sim 2C_{i}'$. 
The pencils are cut out on $Y$ by a pencil of hyperplanes
which have a smooth rational quartic curve
 $R_{i}$ as base locus and the general fiber of the pencil are smooth elliptic curves of degree $6$, c.f. \cite{DolgachevKondoBook}.
 
 In case of a general Enriques surface $Y \subset \PP^{5}$ the pencils are cut out by a pencil of quadrics with a base curve of degree 
$14$ and genus $9$. 
\end{enumerate}

\section{Families of pairs $(Y,C)$}\label{familiesOfPairs}

Let $\calE_{a}$ denote the  irreducible component of the moduli space  of polarized projective surface $(Y,H)$ with $H^{2}=10$, sectional genus $6=\frac{1}{2}H.(H+K)+1$ and canonical divisor $K^2=-a$ for $a=0,\ldots,6$ corresponding to one of our seven families of Table \ref{table1}.
The family of rational surfaces have dimension
$$
\dim \sE_{a}=10+2a
$$
since they depend on the choice of $2(9+a)$ points on $\PP^{2}$ up to the action of $\PGL(3)$.
This formula holds also for the family $\sE_{0}$ of Fano polarized Enriques surface.

Denote by $\sEC_{a}$ the moduli space of pairs $(Y,C)$ where $Y$ is a member of $\calE_{a}$ and $C \in |H|$ is a smooth irreducible curve. There are natural morphisms
$$
\xymatrix{&&\sEC_{a}\ar[dl]_{\varphi_a}\ar[dr]^{\rho_a}\ar[d]^{\chi_a}&\\
&\sE_{a}&\fkP ic^{0}_{{6}}\ar[r]&\fkM_6}
\xymatrix{&&\sEC_{0}\ar[dl]_{\varphi_0}\ar[dr]^{\rho_0}\ar[d]^{\chi_0}&\\
&\sE_{0}&\sR_{6}\ar[r]&\fkM_6.}
$$
for $a>0$ and for $a=0$ where $\sR_{6}$ denotes the moduli space of Prym curves of genus $6$. The map
$\sR_{6}\to \fkM_6$ is a covering of degree $2^{2\cdot 6}-1$.
 The computation in the proof of Theorem \ref{thm1} gives a rational inverse $s\colon \sR_{6} \dasharrow \sEC_{0}$ of $\chi_{0}$. 
 
 \begin{cor}[\cite{CDGK,Verra}]
 The moduli space $\sEC_{0}$ of pairs $(Y,C)$ of a Fano polarized Enriques surfaces $Y$ together with a curve $C \in |H|$ and $
 \sR_{6}$ are birational. In particular $\sR_{6}$ is unirational.
\end{cor}
  
  Thus the computation based on Theorem \ref{extensions} gives a direct proof this result of  \cite{CDGK} which improved the result of \cite{Verra}.
 
 The map $\chi_{a}$ sends a pair $(Y,C)$ to the pair $(C,\eta)$ where $\eta$ is the line bundle of degree $0$ on $C$ which satisfies
 $\omega_{C}\tensor \eta = \sO_{C}(C)$.
 
 The fibers of $\varphi_{a}$ are isomorphic to the open part of smooth curves in $|H|\cong \PP^{5}$. 
 
 \begin{cor} \label{families}The maps $\chi_{a}\colon \sEC_{a} \to \fkP ic^{0}_{6}$ are dominant for $a=3,4,5$. In case of
 \begin{enumerate} 
 \item[$a=5$,] the fibers are $\PP^{4}$'s,
\item[$a=4$,] the Stein factorization of $\chi_{a}$ factors over $\fkP ic^{0}_{6}\times_{\fkM_{6} } \fkM_{6,4}^{1}$ with generic fibers
 $\PP^{2}$'s, where $\fkM_{6,4}^{1}=\{(C,g^{1}_{4})\} \to \fkM_{6}$ denotes the birational $5$-sheeted cover of pairs of a curve of genus $g=6$ together with a $g^{1}_{4}$, and
 \item[$a=3$,] $\sEC_{3}$ is birational to the universal difference variety $\sC^{3,3}$ which is a generically a $20:1$ cover 
 of $\fkP ic_{6}^{0}$.
\end{enumerate}
In particular, $\fkP ic^{0}_{6}$, $\fkM_{6,4}^{1}$ and $\sC^{3,3}$ are unirational.
\end{cor}

 Note that since we make only birational statements we may restrict to the sublocus of $\fkM_{6}$ where a universal curve
 $\sC$ exists. The universal difference variety $\sC^{3,3}=\sC^{(3)}\times_{\sC} \sC^{(3)}$ is defined over this sublocus.

 \begin{proof} In case $a=5$ the surfaces $Y$ in the fiber of $\chi_{a}$ arize from the del Pezzo surface $D$ by blowing up $10$
 points $\Gamma$ on $C \subset D$, and two collections $\Gamma_{1}, \Gamma_{2}$ give a surfaces in the same fiber if and if
 $\Gamma_{1}\sim \Gamma_{2}$. Thus the fibers are isomorphic to $|\Gamma |\cong \PP^{4}$ since 
 $h^{0}(C,\sO_{C}(\Gamma))=10+1-6=5$. For a different argument, by our the computation the maximal extensions $Z$ of $C$, corresponding to the case $a=5$ the fibers are isomorphic to $\PP((H^{0}(Z,\sO_{Z})/H^{0}(C,\sO_{C}(1)))^{*})\cong \PP^{10-5-1}$. 
 
 In case $a=4$ the construction of the $5$ families of $1$-extensions starts with choice of a $g^{1}_{4}$ on $C$. The adjoint surfaces $Y_{1}$ are complete intersection of the corresponding $\PP^{1}\times \PP^{2}\subset \PP^{5}$ with a further quadric in $H^{0}(\PP^{5},\sI_{C}(2))$ thus the choices are $\PP((H^{0}(\PP^{5},\sI_{C}(2))/H^{0}(\PP^{5},\sI_{\PP^{1}\times \PP^{2}}(2)))^{*})\cong \PP^{6-3-1}$. Different choices of the $6$ points $\Gamma$ on $C \subset Y_{1}$  which we have to blow up to get $Y$ from $Y_{1}$, lead generically to different fibers since $h^{0}(C,\sO_{C}(\Gamma))=6+1-6=1$. The same fiber dimensions arizes from the $3$-dimensional quotients
 $H^{0}(Z,\sO_{Z}(1))/H^{0}(C,\sO_{C}(1))$ where $Z$ is one of the five maximal extensions of $C$ of type $a=4$.
 
 In case $a=3$ the surface $Y$ is uniquely determined by representing $\eta=\sO_{C}(p_{1}+p_{2}+p_{3}-q_{1}-q_{2}-q_{3})$
 as a difference divisor.
 
 Thus the connected components of the fibers are generically isomorphic to $\PP^{2(a-3)}$'s and the dimension count
 $\dim \sEC_{a}=10+2a+5=4\cdot6-3+2(a-3)$ fits.
 The unirationality statements follow since $\sE_{a}$ and $\sEC_{a}$ are unirational.
 \end{proof}
 
 \begin{cor} In case of $a=6$ the morphism $\chi_{6}$ has as target the restricted family $\fkP ic^{0}_{6} \times_{\fkM_{6}}\fkM_{6,5}^{2}$
 to the loci of curves which have a smooth plane quintic model and the fibers are isomorphic to the linear system of a divisor $\Gamma$ of $15$ points on $C$:
 $$
 |\Gamma |=\PP^{9}\cong \PP((H^{0}(Z,\sO_{Z}(1)/H^{0}(C,\sO_{C}(1)))^{*}),
 $$
 where $Z=\Sec(\PP^{1}\times\PP^{4}) \subset \PP^{14}$ denotes the maximal extension of the paracanonical curve $C$.
 \end{cor}
 
 \begin{proof} By Theorem \ref{thm2} the surfaces $Y$ are the blow-up of the Veronese surface $Y_{1}=v_{2}(\PP^{2})\subset \PP^{5}$ in a collections
  $\Gamma$ of  $15$ points on $C \subset Y_{1}$ and linear equivalent collections give surfaces in the same fiber $\chi_{6}$.
 The  second description gives a $\PP^{9}$ as well.
 \end{proof}

 \begin{rem} Trigonal and hyperelliptic curves do not occur, since their paracanonical models are contained in quadrics.
The $2\times 2$-minors of the multiplication matrix
 $$
 H^{0}(C,\sO_{C}(D))\times H^{0}(C, \omega_{C}\tensor \eta(-D)) \to H^{0}(C,\omega_{C}\tensor \eta)=H^{0}(\PP^{4},\sO(1))
 $$
 vanish on $C$, 
 where $D$ denotes a trigonal or hyperelliptic divisor. Since $h^{0}(C,\omega_{C}\tensor \eta(-D))\ge 10-3+1-6=2$ by Riemann's
 inequality we get at least one quadric.
 \end{rem}
 
\section{Special curves and surfaces} \label{specialCurves} 

The computations  of Section \ref{maxExtension} proves the following.

\begin{thm}[Reye,\cite{Reye}]
The maximal extension of a general Fano polarized nodal Enriques surface $Y \subset \PP^{5}$ is the $6$-fold
$Z=\Sec(v_{2}(\PP^{3})) \subset \PP^{9}$.
\end{thm}
 
\begin{cor}[Reye,\cite{Reye}] The moduli space $\sE^{nd}$ of nodal Fano polarized Enriques surface is birational to a quotient of a Grassmannian
$$
\GG(6,(H^{0}(\PP^{3},\sO(2)))^{*}) // \SL(4,\CC).
$$
\end{cor}

\begin{rem}
\begin{enumerate}
\item Since $6\cdot 4-15=9$ the group $\SL(4,\CC)$ has at most a finite stabilizer in a general point of the Grassmannian.
\item The $10$ cubics define a self-dual Cremona transformation $\sigma\colon \PP^{9} \dasharrow \PP^{9}$
whose exceptional loci coincides with the quintic hypersurface $\Sec^{3}(v_{2}(\PP^{3})) \subset \PP^{9}$, c.f. 
\cite{EinShepherd-Barron}.
\item The general $Y \in \sE^{nd}$ contains as $20$ half-pencils $F_{i}$ which come in pairs of numerical equivalent 
cubic curves $C_{i},C_{i}'$ which differ by the canonical class. Suitable enumerated  they satisfies
$$ 3H\sim \sum_{i=1}^{10} C_{i}$$
for the hyperplane class $H$ of $Y$.
The elliptic pencils $|2C_{i}|=|2C_{i}'|$ are cut out by a pencil of hyperplanes whose base loci are rational curves of degree
$4$.
For an Enriques surface  $Y$ defined over $\QQ$ the Galois group of the 20 cubics is a subgroup of 
$$((\ZZ/2) ^{10} \rtimes S_{10})\cap A_{20}$$
and for general $Y$ it is the full group, c.f. \cite{DolgachevMarkushevich}.
See \cite{DolgachevCossec,DolgachevKondoBook} for many more information on Enriques surfaces.

\item A special hyperplane section of $\Sec(v_{2}(\PP^{3})) \subset \PP^{9}$ is the essential vision variety, see \cite{FKO}.
The connection stems from the identification of 
$\overline{SO(3)} \subset \PP^{9}$ with $v_{2}(\PP^{3})$ via Study coordinates. 

\item The linear section of $Z$ with a $\PP^{6}$ are Fano-Enriques 3-folds, \cite{CM, CDGK}.
\end{enumerate}
\end{rem}

The obstruction ideal $eq$ of a general curve section $C \subset \PP^{4}$ of a nodal $Y$ decomposes differently 
 then the ideal $eq$ of a general Prym canonical curve. 
 
 \begin{thm} Let $C$ be a general hyperplane section of a general nodal Fano polarized Enriques surface.
 Then $C$ has $7$ families of $1$-extension:
 \begin{enumerate}
 \item[1)] A $4$-dimensional family of $1$-extensions $Y$ with $K_{Y}^{2}=-5$.
 \item[2)] Five $2$-dimensional families of $1$-extensions with $K_{Y}^{2}=-4$.
 \item[4)] A $4$-dimensional family of $1$-extensions to nodal Enriques surfaces $Y$ with $K_{Y}^{2}=0$.
 \end{enumerate}
\end{thm}

\begin{proof} The existence of these seven families is clear. By semi-continuity there are no further families as running he code printed out by the function 
extensionsOfNodalPrymCurve of our package SurfacesAndExtensions \cite{STm2} shows.
If we deforming $C$ into a general Prym canonical curve then the $4$-dimensional family of Enriques extensions drops to a single Enriques surface and 20
surfaces of type \ref{3)}). 
\end{proof}

\begin{rem} In the closure of the family of nodal Enriques surfaces there are Coble surface. Coble surfaces with six singular points corresponding to $(-4)$ rational curves on the desingularization arise by intersecting $\Sec(v_{2}(\PP^{3}))$ with the span of six points in $v_{2}(\PP^{3})$. A general section of a nodal Fano polarized Enriques surface
lies on precisely five Coble surfaces with $4$ singular points. These surface are visible in the free resolution of the ideal $eq$. They can be obtained from the five linear components components of the annihilator of
$\Ext^{7}_S( eq, S)$ where $S$ is the coordinate ring of the $\PP^{m}$ from Theorem \ref{extensions}. 
\end{rem}

\begin{thm} Let $C$ be a general hyperplane section of a general  surface $Y$ with $K_{Y}^{2}=-1$.
 Then $C$ has $27$ families of $1$-extension:
 \begin{enumerate}
 \item[1)] A $4$-dimensional family of $1$-extensions $Y$ with $K_{Y}^{2}=-5$.
 \item[2)] Five $2$-dimensional families of $1$-extensions with $K_{Y}^{2}=-4$.
 \item[3)] Twenty $1$-extensions to a surface with $K_{Y}^{2}=-3$.
 \item[6)] A $1$-extensions to a surface $Y$ with $K_{Y}^{2}=-1$.
 \end{enumerate}
 \end{thm}
 
 \begin{proof} The existence of these families is clear. That there are no further families follows by semi-continuity  and running the code 
 printed by the function extensionsOfSpecialCurves(-1)
 of our package SurfacesAndExtensions \cite{STm2}. 
 \end{proof}

\begin{prop} Let $C$ be a general hyperplane section of a general  surface $Y$ with $K_{Y}^{2}=-2$.
 Then $C$ has at most $23$ families of $1$-extension:
 \begin{enumerate}
 \item[1)] A $4$-dimensional family of $1$-extensions $Y$ with $K_{Y}^{2}=-5$.
 \item[2)] Five $2$-dimensional families of $1$-extensions with $K_{Y}^{2}=-4$.
 \item[3)] Sixteen $1$-extensions to a surface with $K_{Y}^{2}=-3$.
 \item[5)] A $1$-dimensional family of $1$-extensions to surfaces $Y$ with $K_{Y}^{2}=-2$.
 \end{enumerate}
 \end{prop}
 
 \begin{proof} The existence of the families 1), 2) and 5) is clear,  and running the code  printed by function extensionsOfSpecialCurves(-2)
 of our package SurfacesAndExtensions \cite{STm2} establishes that there are at most 16 families of type \ref{3)}) by semi-continuity.
 What in not clear  for more general curves $C$ than the  tested one is, that the number of families of type \ref{3)}) does not drop further. 
 Note that a general
 paracanonical curve has $20$ families of type  \ref{3)}) by Theorem \ref{thm1}. So imposing an extension of type 5) courses some dropping.
 Note that running the code several times (with different RandomSeeds) over moderate size finite ground fields, makes the probability very low, 
 that we have fewer than $23$ families
 for the general $C$ as above.
 \end{proof}
 
For special surfaces we might have larger maximal extensions as in the case of nodal Enriques surfaces. A complete classification of varieties with Betti table (\ref{betTab}) remains open.

\bibliographystyle{plain}

\bibliography{biblioSurfacesAndExtensions}

\end{document}